\documentclass[12pt]{article}
\usepackage{amssymb,amsmath,amsfonts,amsthm}
\usepackage{stmaryrd}
\usepackage{color}
\usepackage[latin1]{inputenc}
\usepackage{enumitem}
\usepackage{textcomp}

\usepackage{algorithm}                   
\usepackage{algorithmic}

\usepackage{hyperref}
\usepackage{booktabs} 
\usepackage{multirow}

\usepackage[small,bf,hang]{caption} 

\usepackage{epsfig}
\usepackage{url,xargs}
\usepackage[font=small,skip=4pt]{caption}

\def\BBox{\kern  -0.2cm\hbox{\vrule width 0.2cm height 0.2cm}}

\newtheorem{theorem}{Theorem}[section]
\newtheorem{lemma}[theorem]{Lemma}

\newtheorem{definition}[theorem]{Definition}
\newtheorem{corollary}[theorem]{Corollary}
\newtheorem{proposition}[theorem]{Proposition}
\newtheorem{remark}[theorem]{Remark}

\newtheorem{conjecture}[theorem]{Conjecture}

\graphicspath{{figures/}}

\setlist{noitemsep}

\newcommand{\N}{\mathbb{N}}

\newcommand{\ca}{\mathcal}

\usepackage[dvipsnames]{xcolor}

\title{\textbf{An algorithm and new bounds for the circular flow number of snarks}}

\date{}

\author{ Jan Goedgebeur$^{1,2}$\thanks{Supported by a Postdoctoral Fellowship of the Research Foundation Flanders (FWO)}, Davide Mattiolo$^3$, Giuseppe Mazzuoccolo$^4$\\ 
\footnotesize $^1$ Deparment of Applied Mathematics, Computer Science and Statistics,\\[-3mm]
\footnotesize Ghent University, Krijgslaan 281 - S9, 9000 Ghent, Belgium.\\
\footnotesize $^2$ Computer Science Department,\\[-3mm]
\footnotesize University of Mons, Place du Parc 20, 7000 Mons, Belgium.\\
\footnotesize $^3$ Dipartimento di Scienze Fisiche, Informatiche e Matematiche,\\[-3mm]
\footnotesize Universit\`{a} di Modena e Reggio Emilia, Via Campi 213/b, 41126 Modena, Italy.\\
\footnotesize $^4$ Dipartimento di Informatica,\\[-3mm]
\footnotesize Universit\`{a} degli Studi di Verona, Strada le Grazie 15, 37134 Verona, Italy.\\
}
\begin{document}
\maketitle

\vspace{-1em}

\begin{abstract}
It is well-known that the circular flow number of a bridgeless cubic graph can be computed in terms of certain partitions of its vertex-set with prescribed properties. 
In the present paper, we first study some of these properties that turn out to be useful in order to design a more efficient algorithm for the computation of the circular flow number of a bridgeless cubic graph. Using this algorithm, we determine the circular flow number of all snarks up to 36 vertices as well as the circular flow number of various famous snarks.
After that, as combination of the use of our algorithm with new theoretical results, we present an infinite family of snarks of order $8k+2$ whose circular flow numbers meet a general lower bound presented by Lukot'ka and \v Skoviera in 2008. In particular this answers a question proposed in their paper.
Moreover, we improve the best known upper bound for the circular flow number of Goldberg snarks and we conjecture that this new upper bound is optimal.
Finally, we discuss a possible extension to the computation of the circular flow number in the non-regular case.
\end{abstract}



\section{Introduction}\label{intro}

Given a real number $r\ge2$, a circular nowhere-zero $r$-flow in a graph $G=(V,E)$ is an orientation of $G$ together with a flow function $f\colon E\to [1,r-1]$ such that, at every vertex, the sum of all incoming flow values equals the sum of all outgoing ones. The circular flow number of a graph $G$, denoted by $\Phi_c(G)$, is the least real number $r$ such that $G$ admits a circular nowhere-zero $r$-flow. This parameter was explicitely introduced in~\cite{GTZ} and shown to be a minimum and a rational number for every bridgeless graph. Tutte's $5$-flow Conjecture~\cite{Tutte54} is one of the most important and outstanding conjectures in the theory of flows in graphs. The conjecture claims that every bridgeless graph admits a (circular) nowhere-zero $5$-flow and it is well-known that it is equivalent to its restriction to cubic graphs. When dealing with such graphs, the circular flow number is strictly related to the chromatic index, namely a cubic graph $G$ is $3$-edge-colorable if and only if $\Phi_c(G)\le4$. (The chromatic index of a graph is the minimum number of colors required for a proper edge-coloring of the graph.)
Therefore a counterexample to Tutte's conjecture, if any, must be found inside the class of non-$3$-edge-colorable cubic graphs. In particular, it is known that such a counterexample must be found inside the class of snarks, i.e.\ cyclically $4$-edge-connected cubic graphs with girth at least $5$ which do not admit a $3$-edge-coloring. In contrast to $3$-edge-colorable cubic graphs, whose circular flow number can be either $3$ or $4$, it was proved in~\cite{LuSk} that for every rational number $r \in (4,5]$ there is a snark $G$ such that $\Phi_c(G)=r$. 
Due to Tutte's $5$-flow Conjecture, snarks with circular flow number exactly $5$ have a certain interest and they were studied in~\cite{EMT16, GMM, MaRa06}.

One of the aims of this paper is to design an algorithm that computes the circular flow number of a cubic graph: our algorithm works for all bridgeless cubic graphs having circular flow number strictly less than $5$ without additional assumptions; if a bridgeless cubic graph has circular flow number at least $5$, the algorithm only says that it has $\Phi_c(G) \geq 5$. Clearly, if Tutte's $5$-flow Conjecture holds true, then it can be applied to all bridgeless cubic graphs.

It works using a well-known relation between nowhere-zero flows and bisections in cubic graphs, see for instance~\cite{EMT_bisections}. Note also that an equivalent formulation in terms of balanced valuations is given for instance in~\cite{Jaeg}. 

Section~\ref{sec:properties} is devoted to the description of properties of bisections that turn out to be useful for the design of our algorithm to determine the circular flow number of a (bridgeless) cubic graph, which is presented in Section~\ref{sec:algorithm}. Using our implementation of this algorithm, we determined the circular flow number of all snarks up to 36 vertices as well as the circular flow number of various famous snarks. The results of these computations can also be found in Section~\ref{sec:algorithm}.

Two of the main results are given in Section~\ref{sec:new_bounds}. In this section we first improve the previous known upper bound from~\cite{Luk} for the circular flow number of the Goldberg snarks. Moreover, in \cite{Luk_cycle}, Lukot'ka and \v Skoviera proved a general lower bound for the circular flow number of a snark in terms of its order and, at the end of their paper, they suggest that there might exist an infinite family of snarks of order $8k+2$ whose circular flow numbers meet the general lower bound presented. Here, we confirm this by constructing such an infinite family.
Section~\ref{Sec:certificatecircularflownumber5} is devoted to a more theoretical result which shows that the problem of deciding if $\Phi_c(G)\ge5$, for a non-cubic graph $G$, can be reduced to deciding if $\Phi_c(H)\ge5$, for every $H\in \ca H$, where $\ca H$ is a finite class of cubic graphs. 
The paper ends with Section~\ref{sec:open_problems} where we present two new conjectures about the circular flow number of snarks.

\section{Useful Properties of Good Bisections}
\label{sec:properties}

A $2$-bisection of a graph $G$ is a partition of its vertex set $V(G)= \ca B \cup \ca W$ such that $|\ca B|=|\ca W|$ and each connected component of both induced subgraphs $G[\ca B]$ and $G[\ca W]$ has at most $2$ vertices. We will refer to the vertices of $\ca B$, resp.\ $\ca W$, as \textbf{black}, resp.\ \textbf{white}, vertices. We can define the set $\partial(X) = \{ uv\in E(G)\colon u\in X\text{ and } v\notin X \}$ and $\Delta(X)=|b_X-w_X|$, where $b_X=|\ca B \cap X|$ and $w_X=|\ca W \cap X|$.

A $2$-bisection is said \textbf{orientable} if and only if \[\frac{|\partial(X)|}{\Delta(X)} \ge 1\] for every $X\subseteq V(G)$.

It is easy to check that for any $r<5$, any circular nowhere-zero $r$-flow of a cubic graph $G$ induces a $2$-bisection of $G$: color a vertex white or black according to the number of inner edges ($1$ or $2$, respectively) in the orientation of $G$ corresponding to the flow with a positive value for every edge.  
It is well-known, see for instance~\cite{EMT_bisections} and~\cite{Jaeg}, that, if the circular flow number of $G$ is less than $5$, in order to compute the circular flow number of $G$ we can search for subsets $X$ of $V(G)$ that minimize the ratio \begin{equation}\label{ratio}
\frac{|\partial(X)|}{\Delta(X)}
\end{equation}
when a $2$-bisection is fixed and then search for the maximum among these values.

The well-known relation between the ratio~(\ref{ratio}) and the circular flow number of $G$, if $\Phi_c(G)<5$, is the following.

\begin{equation}\label{ratiovscfn}
\max_{\text{$2$-bisection of $G$}} \left( \min_{X \subset V(G)} \frac{|\partial(X)|}{\Delta(X)} \right)= \frac{\Phi_c(G)}{\Phi_c(G)-2}
\end{equation}

We would like to stress that if $G$ has circular flow number at least $5$, it is not true in general that its flow naturally induces a $2$-bisection. For instance, the Petersen graph does not admit a $2$-bisection at all, and there exist other bridgeless cubic graphs, admitting a $2$-bisection,  with the property that every $5$-flow does not induce a $2$-bisection (see for instance~\cite{AGLM, EMT_bisections, Tarsi} for a more general discussion about bisections in cubic graphs).

Anyway, for any fixed $2$-bisection of a graph, if it does exist, we call the subsets that minimize the ratio~(\ref{ratio}) \textbf{good}. Moreover, we call the $2$-bisections that maximize the left term in~(\ref{ratiovscfn}) \textbf{optimal}.

If $\Delta(X)=0$, we define its ratio to be $\infty$, hence we will look for subsets of $V(G)$ such that $\Delta(X)>0.$ In particular, if $X$ is a proper subset of $V(G)$ it follows that $\frac{|\partial(X)|}{\Delta(X)}= \frac{|\partial(\bar{X})|}{\Delta(\bar{X})}$, where $\bar{X}$ denotes $V(G)-X$, and so, for a given $2$-bisection, we can always find at least a good subset of order at most $\frac{|V(G)|}{2}$. 
From now on, we will also assume without loss of generality to have more black vertices than white ones in $X$, i.e.\ $b_X> w_X.$

\begin{lemma}\label{lem:conn}
	Consider a graph $G$ having a $2$-bisection $V(G)=\ca B\cup \ca W$ and a subset $X\subseteq V(G)$. Suppose that there is $X\subseteq V$ such that $G[X]$ is disconnected with components $A_1,\dots, A_n$. Then there is one of those components $A$ such that \[ \frac{|\partial(A)|}{\Delta(A)} \le \frac{|\partial(X)|}{\Delta(X)}. \]
\end{lemma}

\begin{proof}
	There is $A\in \{A_1,\dots,A_n\}$ such that $\frac{|\partial(A)|}{\Delta(A)}\le \frac{|\partial(A_i)|}{\Delta(A_i)}$ for each $i$. Therefore from $|\partial(A)|\Delta(A_i)\le|\partial(A_i)|\Delta(A)$ and summing up all such inequalities we get \[ \frac{|\partial(A)|}{\Delta(A)}\le \frac{\sum_{i=1}^n|\partial(A_i)|}{\sum_{i=1}^n\Delta(A_i)} \le \frac{|\partial(X)|}{\Delta(X)}. \]
\end{proof}

Applying the previous lemma we can conclude that if $X\subseteq V$ is a good subset such that $G[X]$ is not connected, then all its connected components are good as well.

Consider a graph $G$ with a $2$-bisection. For a subset $X\subseteq V(G)$ let us denote by $\partial_V(X) = \{ v \in X \colon \deg_{G[X]}(v)<3\} =\{ v\in X\colon \exists\, w\in V(G)-X \text{ such that } vw \in \partial(X) \}$.

\begin{lemma}\label{lem:monochrom}
	Consider a bridgeless cubic graph $G$ having a $2$-bisection. Consider a $2$-bisection $V(G)=\ca B\cup \ca W$ and one of its good subsets $X\subseteq V(G)$, with $b_X> w_X$ and $\frac{|\partial(X)|}{\Delta(X)}>1$.
	Then, $\partial_V(X)$ is a subset of black vertices.
\end{lemma}

\begin{proof}
	We want to show that there are no white vertices in $\partial_ V(X)$. Suppose by contradiction that there is a white vertex $v\in\partial_V(X)$.
	
	If $v$ is incident to a unique edge of $\partial(X)$. Then, if we set $Y:= X-v$, \[ \frac{|\partial(Y)|}{\Delta(Y)} = \frac{|\partial(X)|+1}{\Delta(X)+1} < \frac{|\partial(X)|}{\Delta(X)} \] which is a contradiction since $X$ is good. 
	
	If, on the other hand, $v$ is incident to two edges of $\partial(X)$ then, if we set $Y:= X-v$, \[ \frac{|\partial(Y)|}{\Delta(Y)} = \frac{|\partial(X)|-1}{\Delta(X)+1} < \frac{|\partial(X)|+1}{\Delta(X)+1} < \frac{|\partial(X)|}{\Delta(X)} \] and again we have a contradiction.
	
\end{proof}

\begin{remark}\label{remark:monochrom}
	If $X\subseteq V(G)$ is a good subset of vertices in a $2$-bisection, then also $\bar{X}$ is good. In particular, if the $2$-bisection is optimal then both $\partial_V(X)$ and $\partial_V(\bar{X})$ are monochromatic (in particular if one is white the other is black).
\end{remark}

\begin{corollary}\label{cor:monochrom}
	Consider a bridgeless cubic graph $G$ with $\Phi_c(G)<5$. Consider an optimal $2$-bisection $V(G)=\ca B\cup \ca W$, and let $X\subseteq V(G)$ be a good subset. Then there is no couple of adjacent vertices $v,w$ with the same color such that \[ v\in X \text{ and } w\in \bar{X}. \]
\end{corollary}

\begin{remark}\label{summary_properties}
	We have proved that, for a given optimal $2$-bisection of a bridgeless cubic graph $G$ with circular flow number less than $5$ and among all \textbf{good} subsets of vertices \begin{itemize}
		\item there is at least one of them $X$ that induces a connected subgraph; 
		\item we can search it among all subsets with cardinality at most $\frac{|V(G)|}{2}$, since the ratio of a subset equals the ratio of its complement;
		\item the boundaries $\partial_V(X)$ and $\partial_V(\bar{X})$ are monochromatic of different colors.
	\end{itemize}
\end{remark}
	
The main idea of the algorithm presented in the following section is to only process sets $X$ which satisfy the three properties in Remark~\ref{summary_properties}. In order to assure that this produces consistent results, we need to stress that in every $2$-bisection (not necessarily optimal) there exists a set $X$ (not necessarily good) which satisfies all three properties and such that the ratio $\frac{|\partial(X)|}{\Delta(X)}$ is less than or equal to the ratio for a good set in an optimal $2$-bisection.
The critical property is the one on monochromatic boundaries, since it follows by Lemma~\ref{lem:monochrom} where we need to assume $\frac{|\partial(X)|}{\Delta(X)}>1$. Indeed, in principle, it could be that in a non-orientable $2$-bisection all good subsets do not satisfy the third property in Remark~\ref{summary_properties} and it could be that, at the same time, all subsets satisfying such properties have a ratio larger than the minimum one in an optimal $2$-bisection. The following lemma excludes this possibility.

\begin{lemma}
	Consider a bridgeless cubic graph $G$ having a $2$-bisection. Consider a $2$-bisection $V(G)=\ca B\cup \ca W$ and a good subset $X\subseteq V(G)$, with $b_X> w_X$ and $\frac{|\partial(X)|}{\Delta(X)}\leq 1$.
	Then, there exists a subset $X'$ of $X$ such that $\partial_V(X')$ is a subset of black vertices and $\frac{|\partial(X')|}{\Delta(X')}\leq 1$.
\end{lemma}

\begin{proof}
 If $\partial_V(X)$ is a subset of black vertices, then trivially we can take $X'=X$. 
 	Assume there is a white vertex $v$ in $\partial_ V(X)$. 
	
	If $v$ is incident to a unique edge of $\partial(X)$. Then, since $\frac{|\partial(X)|}{\Delta(X)}\leq1$: 
	\[ \frac{|\partial(X-v)|}{\Delta(X-v)} = \frac{|\partial(X)|+1}{\Delta(X)+1} \leq 1. \]
	
	If, on the other hand, $v$ is incident to two edges of $\partial(X)$ then, \[ \frac{|\partial(X-v)|}{\Delta(X-v)} = \frac{|\partial(X)|-1}{\Delta(X)+1} < \frac{|\partial(X)|+1}{\Delta(X)+1} \leq 1 . \] 

By repeating removal of vertices in this way, we obtain a subset $X'$ of $X$ which satisfies the required properties.	
\end{proof}

\section{Algorithm and Computational Results}\label{sec:algorithm}

The pseudocode of our algorithm to compute the circular flow number of a bridgeless cubic graph is shown in Algorithm~\ref{algo:cfn}. The notation and definitions of $\partial(X)$ and $\Delta(X)$ are as in Section~\ref{sec:properties}. Furthermore, we also use several properties of \textit{good} subsets from the previous section to speed up the algorithm (cf.\ Remark~\ref{summary_properties}). It is also possible to give an optional input parameter $r$ to the algorithm in case you only want to know if $\Phi_c(G) \geq r$ or not. This is usually faster than computing the exact value of $\Phi_c(G)$.

\begin{algorithm}[hbt!]
\caption{Compute the circular flow number of a (bridgeless) cubic graph $G$}

\label{algo:cfn}
  \begin{algorithmic}
   \STATE \textbf{Optional input:} value for $r$
   \IF{$r$ is defined}
   		\STATE \verb|test_lower_bound| := 1 \ // i.e.\ only test if $\Phi_c(G) \geq r$
   	\ELSE
   		\STATE \verb|test_lower_bound| := 0 \ // i.e.\ compute $\Phi_c(G)$
   \ENDIF
   \STATE \verb|max_min_fraction| := 0
 \FOR{every 2-bisection $(\ca B, \ca W)$ of $G$}   
 	\STATE \verb|min_fraction| := $\infty$
 	\FOR{every subset $X \subseteq V(G)$ for which: $2 \leq |X| \leq \frac{|V(G)|}{2}$ \AND $G[X]$ is connected  \AND $\partial_V(X)$ and $\partial_V(\bar{X})$ are monochromatic of different colors}
 	\STATE Compute $|\partial(X)|$ and $\Delta(X)$
 	\IF{$\frac{|\partial(X)|}{\Delta(X)} <$ \texttt{min\_fraction} }
	 	\STATE \verb|min_fraction| := $\frac{|\partial(X)|}{\Delta(X)}$
	 	\IF{\texttt{test\_lower\_bound} \AND \texttt{min\_fraction} $\leq \frac{r}{r-2}$ }
	 	 	\STATE abort subset search\ // since we are searching for a \verb|min_fraction| $> \frac{r}{r-2}$
	 	 \ENDIF
 	\ENDIF
 	\ENDFOR
	\IF{\texttt{test\_lower\_bound} \AND \texttt{min\_fraction} $> \frac{r}{r-2}$ }
	 	\RETURN $\Phi_c(G) < r$
	 \ENDIF 	
 	\IF{\texttt{min\_fraction} $>$ \texttt{max\_min\_fraction}}
 		\STATE \verb|max_min_fraction| := \verb|min_fraction|
 	\ENDIF	 
 \ENDFOR
	\IF{\texttt{test\_lower\_bound}} 
	 	\RETURN $\Phi_c(G) \geq r$ \ // i.e.\ $\texttt{max\_min\_fraction} \leq \frac{r}{r-2}$
	\ELSE
		\RETURN  $\Phi_c(G) = \frac{2 \cdot \texttt{max\_min\_fraction}}{\texttt{max\_min\_fraction}-1}$
	 \ENDIF 	 
  \end{algorithmic}
\end{algorithm}

We implemented Algorithm~\ref{algo:cfn} in the programming language C. The source code of the program can be obtained from~\cite{cfnwebsite}. In~\cite{BGHM} Brinkmann et al. determined all snarks up to 36 vertices. Using our algorithm, we determined all snarks of circular flow number 5 up to 36 vertices in~\cite{GMM}. We now also determined the circular flow number of all other snarks up to 36 vertices and the results can be found in Table~\ref{table:cfn_values}.

\begin{table}[h!tb]
	\centering

	\begin{tabular}{c  rrrrr r}
		\toprule	
		 \multirow{2}{*}{Order} & \multicolumn{6}{c}{Circular flow number} \multirow{2}{*}{Total} \\
 &  $4+1/4$ & $4+1/3$ & $4+1/2$ & $4+2/3$ & $5$ & \\
		\midrule
10 &  &  &  &  & 1 & 1 \\
18 &  &  & 2 &  &  & 2 \\
20 &  &  & 6 &  &  & 6 \\
22 &  &  & 20 &  &  & 20 \\
24 &  &  & 38 &  &  & 38 \\
26 &  & 57 & 223 &  &  & 280 \\
28 &  & 1 258 & 1 641 &  & 1 & 2 900 \\
30 &  & 10 500 & 17 897 &  & 2 & 28 399 \\
32 &  & 60 008 & 233 042 &  & 9 & 293 059 \\
34 & 3 627 & 372 708 & 3 457 227 &  & 25 & 3 833 587 \\
36 & 199 338 & 3 339 506 & 56 628 773 & 17 & 98 & 60 167 732 \\
		\bottomrule		
	\end{tabular}
	\caption{The values of the circular flow number of all snarks up to 36 vertices.}
	\label{table:cfn_values}
\end{table}

Using our algorithm, we also determined the circular flow number of various famous named snarks. The results are summarized in Table~\ref{table:cfn_famous_snarks} together with the circular flow number of the Flower snarks, which was already determined by Lukot'ka and \v Skoviera in~\cite{Luk_cycle},  of the Generalized Blanu\v sa snarks, which was already determined by Lukot'ka in~\cite{Luk}, and of some Goldberg snarks. We remark that the circular flow number of the Goldberg snarks of order $16k+8$ is not completely determined  for $k\ge4$. It is only known to belong to the interval $[4+\frac{1}{2k+1},4+\frac{1}{k+1}]$, see Section~\ref{Sec:Goldberg} for details.

\begin{table}[h!tb]
	\centering
	\begin{tabular}{rrl}
		\toprule	
Name & Order & $\Phi_c$ \\
		\midrule
(Generalized) Blanu\v sa snarks \cite{Blanusa}, \cite{Watkin} & $8k+2$ & $4+1/2$ \cite{Luk} \\
Flower snark $J_{2k+1}$~\cite{Isaacs} & $8k+4$ & $4+1/k$~\cite{Luk_cycle} \\
Goldberg snark $G_3$~\cite{Goldberg} & $24$ & $4+1/2$ \\
Goldberg snark $G_5$~\cite{Goldberg} & $40$ & $4+1/3$ \\
Goldberg snark $G_7$~\cite{Goldberg} & $56$ & $4+1/4$ \\

Goldberg snark $G_{2k+1}$~\cite{Goldberg} & $16k+8$ & $[4+1/(2k+1),4+1/(k+1)]$ \\

Loupekine snark 1 and 2~\cite{Loupek} & $22$ & $4+1/2$ \\

Celmins-Swart snarks 1 and 2~\cite{CelStw} & $26$ & $4+1/2$ \\

Double star snark~\cite{Isaacs} & $30$ & $4+1/3$ \\

Szekeres snark~\cite{Szek} & $50$ & $4+1/2$ \\

Watkins snark~\cite{Watkin} & $50$ & $4+1/3$ \\
		\bottomrule		
	\end{tabular}
	\caption{The values of the circular flow number of various famous snarks.}
	\label{table:cfn_famous_snarks}
\end{table}

\section{Improving Bounds for the Circular Flow Number of Some Snarks}
\label{sec:new_bounds}

\subsection{Snarks having minimum possible circular flow number}

In~\cite{Luk_cycle} a lower bound on the circular flow number that depends only on the order of a graph is given, that is:

\begin{theorem}[Lukot'ka and \v Skoviera~\cite{Luk_cycle}]\label{thm:luksko}
	Let $G$ be a connected bridgeless cubic graph of order at most $8k+4$ that does not admit any $3$-edge-coloring. Then \[ \Phi_c(G)\ge 4+\frac{1}{k}. \]
\end{theorem}

In the same paper it is shown that Flower snarks form a family of snarks of order $8k+4$ that attain this bound with equality, more precisely the Flower snark $J_{2k+1}$ has $8k+4$ vertices and circular flow number $4+\frac{1}{k}$, which shows that the upper bound given in \cite{Steffen} was indeed the optimal one.

The paper also reports that Edita M\'{a}\v{c}ajov\'{a} (using a computer search from~\cite{MaRa06}) determined that the two Blanu\v sa snarks on $18 = 8\cdot 2 + 2$ vertices have circular flow number $4+\frac{1}{2}$ and that there are exactly 57 snarks on $26 = 8\cdot 3 + 2$ vertices with circular flow number $4+\frac{1}{3}$. In~\cite{Luk_cycle} Lukot'ka and \v Skoviera mention that this strongly suggests that there exists an infinite family of snarks of order $8k+2$ with circular flow number $4+\frac{1}{k}$. They also report that they are not aware of any graphs of order $8k$ or $8k-2$ with circular flow number $4+\frac{1}{k}$.

In Table~\ref{table:cfn_values} from the previous section, we determined the circular flow number of all snarks up to 36 vertices. The graphs from Table~\ref{table:cfn_values} with the minimum circular flow number for each order
can be downloaded from the \textit{House of Graphs}~\cite{hog} at \url{http://hog.grinvin.org/Snarks}. As can be seen from that table, none of the snarks up to 36 vertices of order $8k$ or $8k-2$ has circular flow number $4+\frac{1}{k}$.

We now present, for every positive integer $k$, a family $\ca S =\{S_k\}$ consisting of snarks of order $8k+2$ and having circular flow number $4+\frac{1}{k}$.
Every snark $S_k$ is obtained by performing a dot product of $S_{k-1}$ and a copy of the Petersen graph with two adjacent vertices removed in a suitable way. We recall the definition of dot product of two connected cubic graphs, say $G$ and $H$, on at least $6$ vertices. Consider $G'= G - \{ab, cd\}$, where $ab$ and $cd$ are independent edges of $G$. Let $H'=H-\{x,y\}$, where $x$ and $y$ are adjacent vertices in $H$, and let $u,v$ and $w,z$ be the other two neighbours of $x$ and $y$, respectively. Then the dot product $G \cdot H$ is defined as the graph
$$(V (G) \cup V (H'), E(G') \cup E(H') \cup \{au , bv , cw , dz \}).$$

Indeed, if the way in which vertices $a,b,c,d$ and $u,v,w,z$ are linked is not specified, there are several ways to form the dot product for selected edges $ab$, $cd$ and vertices $x$ and~$y$. This order will be relevant in our construction as only one specific way seems to work for our aims.

We inductively define the snark $S_k$ as follows:

\begin{itemize}
	\item Let $S_1$ be the Petersen graph.
	\item Let $S_2$ be the Blanu\v sa snark obtained by performing the dot product between two copies $P_1$ and $P_2$ of the Petersen graph where, we select a pair of edges of $P_1$ at distance $1$ (where by distance we mean the number of edges of the shortest path connecting two ends of those edges) and a pair of adjacent vertices of $P_2$.
	\item For $k\ge3$, $S_k$ is obtained with a dot product of $S_{k-1}$ and a copy of the Petersen graph where we select a pair of adjacent vertices of the Petersen graph (by symmetry every pair) and the two independent edges of $S_k$ as indicated in Figure~\ref{fig:Sk} (bold edges).
\end{itemize}

\begin{figure}[htb!]
			\centering
			\includegraphics[width=14cm]{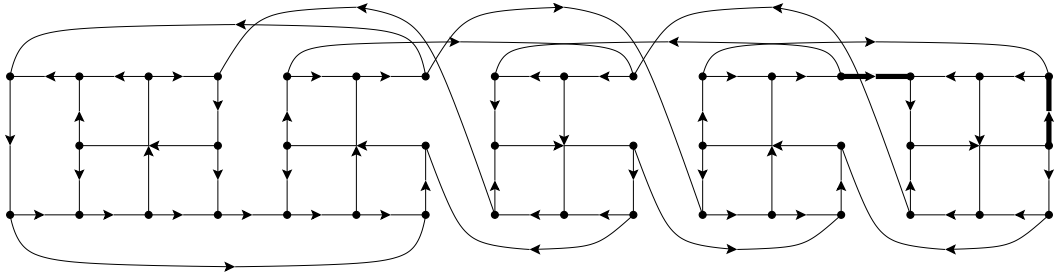}
	\caption{The Snark $S_5$.}\label{fig:Sk}
\end{figure}

It is easy to check that $S_2$ has order $18=8\cdot2+2$ and that it has circular flow number  $4+\frac{1}{2}$. 

Now, we assign to every $S_k$ an orientation as in Figure~\ref{fig:Sk}: the generalization to an arbitrary value of $k$ is clear when we note that any new block with 8 vertices (obtained by dot product with a new copy of the Petersen graph) has the opposite orientation with respect to the last block in $S_{k-1}$.

\begin{theorem}
For any positive integer $k$, $S_k$ is a snark of order $8k+2$ with circular flow number $4+\frac{1}{k}$.
\end{theorem}
\begin{proof}
It is well-known that the dot product of two snarks is a snark, and an easy computation shows that the order of $S_k$ is equal to $8k+2$.
Theorem~\ref{thm:luksko} gives the lower bound $4+\frac{1}{k}$ for all snarks of order $8k+2$. Then, we only need  to show that a nowhere-zero flow with maximum flow value $3+\frac{1}{k}$ can be defined in $S_k$. 
We construct such a flow in the following way. Firstly, we exhibit a $4$-flow $f$ of $S_k$ which has flow value zero only for a specific edge $e$ (the dashed edge in Figure~\ref{fig:Sk4flow}). Then, we construct a set of $k$ oriented (with respect to the given orientation) cycles, say $C_1,\ldots,C_k$, in $S_k$ such that:
\begin{itemize}
\item the edge $e$ belongs to every $C_i$;
\item every edge of $S_k$ having flow value $3$ in $f$ belongs to exactly one of the cycles $C_i$'s.
\end{itemize}

Such properties assure that if we construct a flow $f'$ starting from $f$ and by adding a flow equal to $\frac{1}{k}$ along every oriented cycle $C_i$ then we obtain a nowhere-zero $(4+\frac{1}{k})$-flow of $S_k$. Indeed, the former property implies that the edge $e$ has flow value $k \cdot \frac{1}{k}=1$ in $f'$, and the latter one implies that every other edge has flow value in the interval $[1,3+\frac{1}{k}]$.

\begin{figure}[htb!]
			\centering
			\includegraphics[width=14cm]{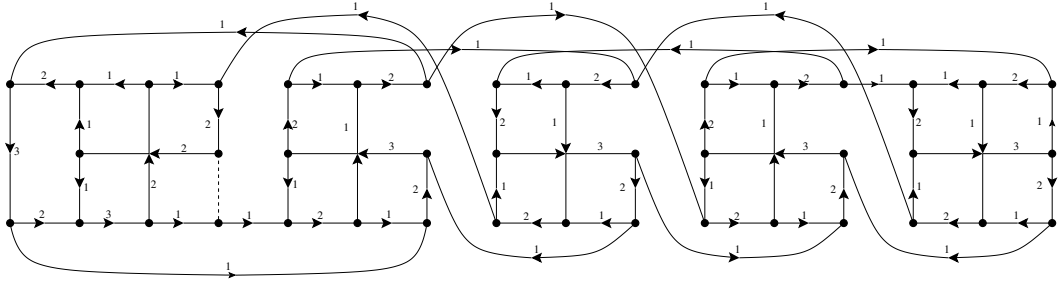}
	\caption{A $4$-flow in $S_5$: the dashed edge is the unique one with flow value zero.}\label{fig:Sk4flow}
\end{figure}

Figure~\ref{fig:Skcycles} shows the set of five cycles for the case $k=5$. 
We refer to this example to briefly explain the general construction of the cycles $C_1,\ldots,C_k$. 

\begin{figure}[htb!]
			\centering
			\includegraphics[width=12cm]{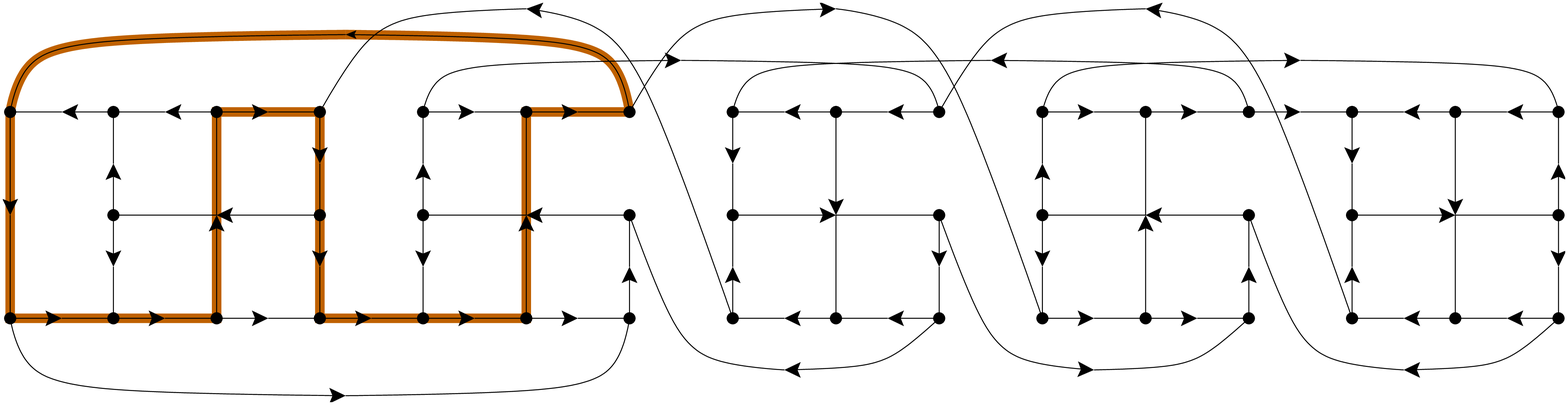}
			
			\includegraphics[width=12cm]{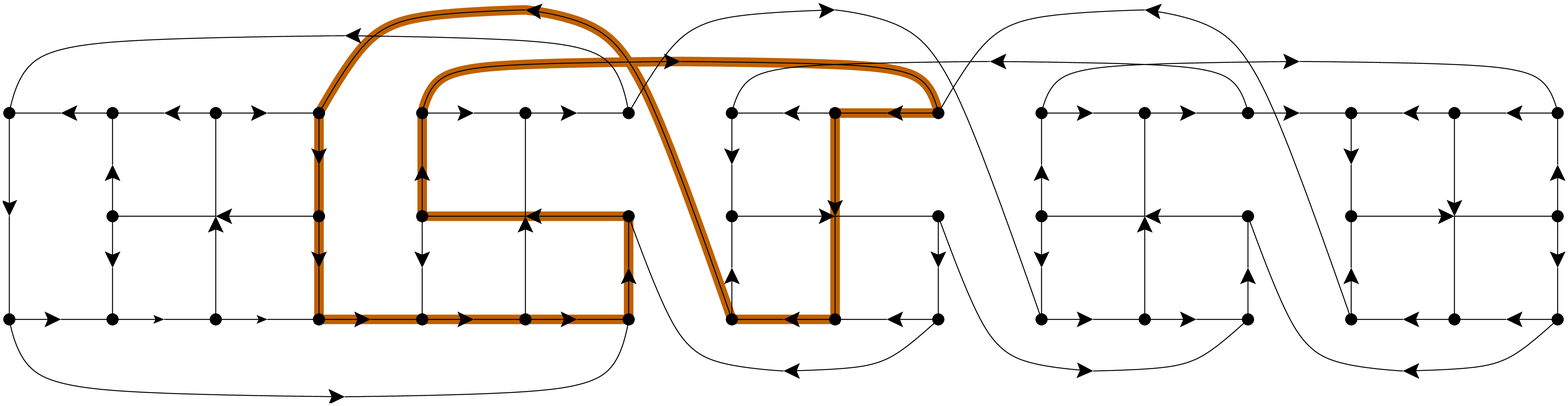}
			
			\includegraphics[width=12cm]{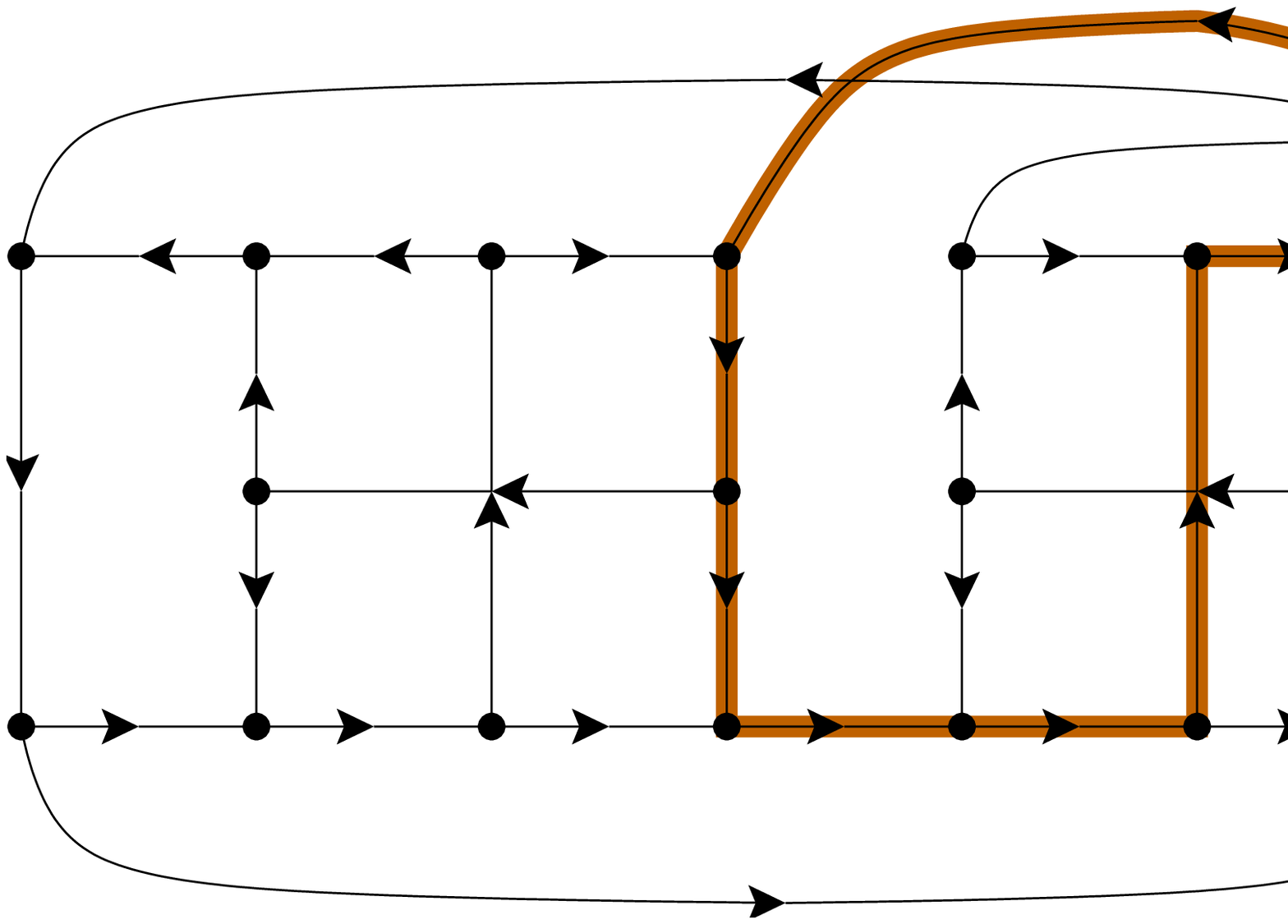}
			
			\includegraphics[width=12cm]{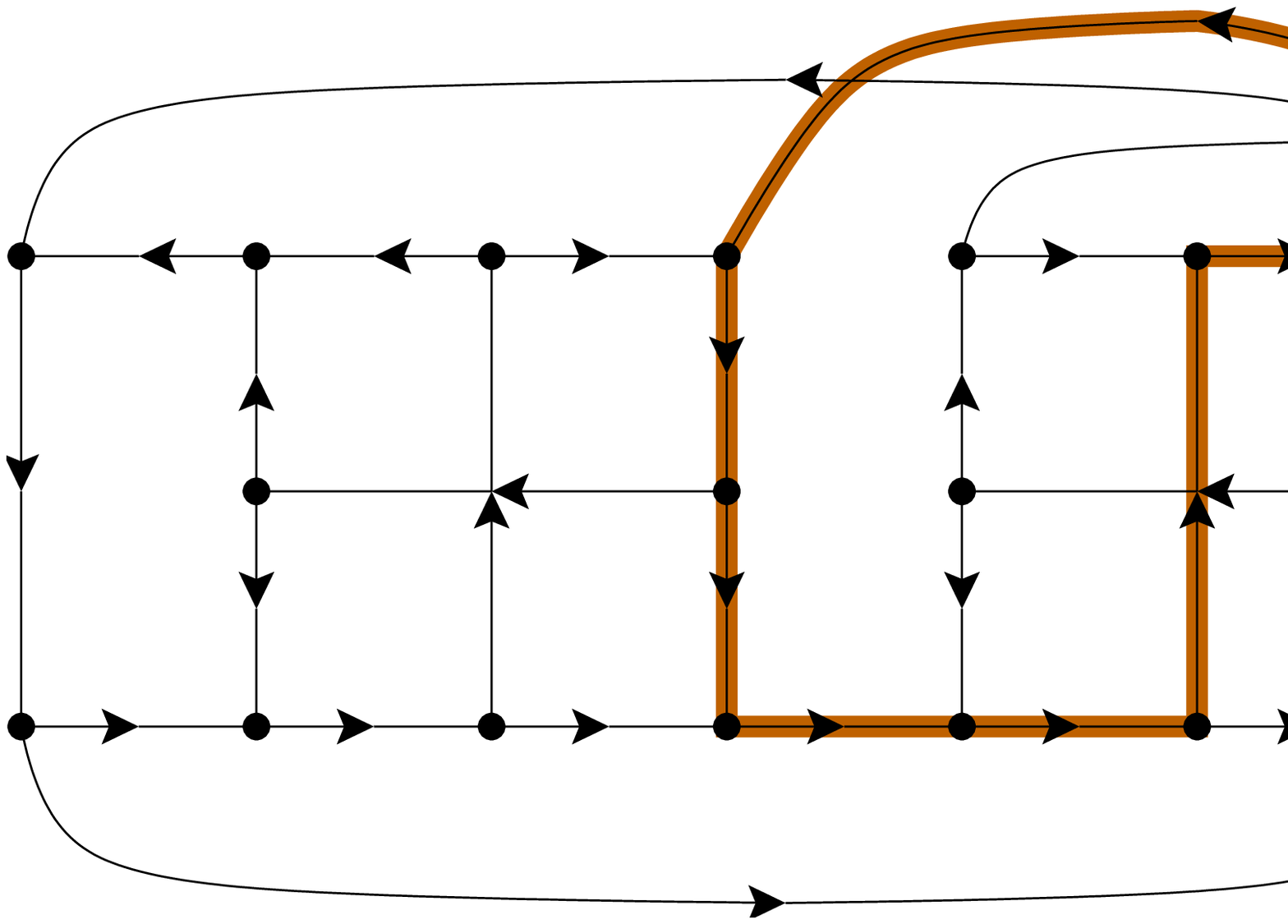}
			
			\includegraphics[width=12cm]{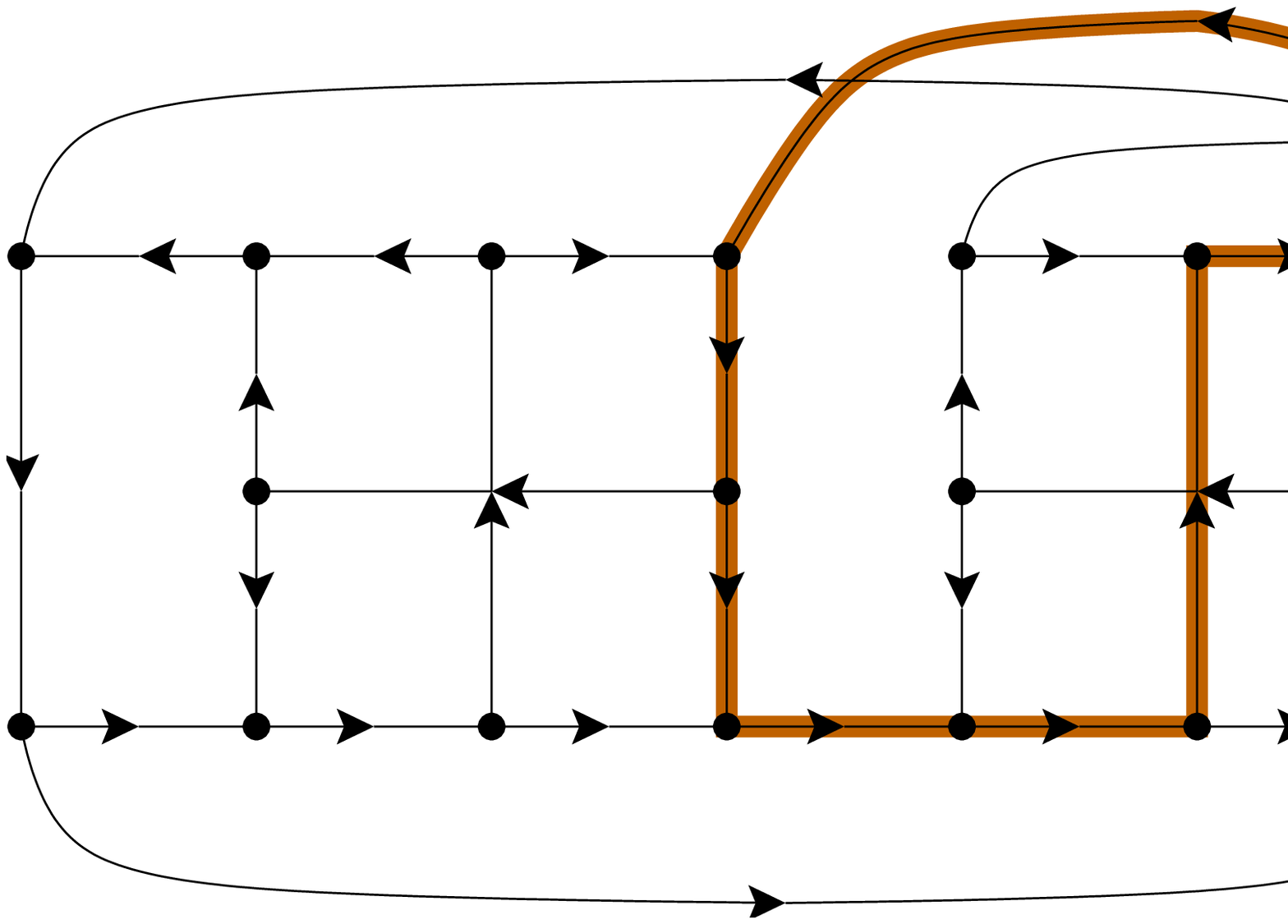}
	\caption{Cycles $C_1,C_2,C_3,C_4$ and $C_5$ in $S_5$.}\label{fig:Skcycles}
\end{figure}

For any $k>1$, the cycle $C_1$ of $S_k$ is the highlighted cycle in the first graph of Figure~\ref{fig:Skcycles}. Indeed, note that $C_1$ does not depend on how many times a dot product is performed to obtain $S_k$. Moreover, it contains the two edges with flow value $3$ in $f$ in the leftmost copy of the Petersen graph.

Every other $C_i$, for $1<i<k$, is constructed analogously as it is shown for the second, third and fourth cycle in Figure~\ref{fig:Skcycles}. In particular, note that also in this case $C_i$ does not depend on the value of $k$, if $k>i$, and the unique edge of $S_k$ with flow value $3$ in the cycle $C_i$ is the horizontal edge in the $i^{th}$ copy of the Petersen graph. 

Finally, we construct the last cycle $C_k$ in analogy to the construction of the fifth cycle in Figure~\ref{fig:Skcycles}. Again the unique edge with flow value $3$ in the cycle $C_k$ is the horizontal edge in the rightmost copy of the Petersen graph.

All of these $k$ oriented cycles pass through the dashed edge $e$ in Figure~\ref{fig:Sk4flow} and, as remarked, every edge of $S_k$ with flow value $3$ belongs to exactly one of them. Then, we can construct a nowhere-zero $(4+\frac{1}{k})$-flow of $S_k$ and the assertion follows.
\end{proof}

For the sake of completeness, we remark that all snarks $S_k$ constructed here are permutation snarks of order $8k+2$, i.e.\ $S_k$ admits a $2$-factor consisting of two chordless cycles of length $4k+1$. 

\subsection{A new upper bound for the Circular Flow Number of the Goldberg Snarks}\label{Sec:Goldberg}
\label{sec:goldberg}

The Goldberg snarks $\{G_{2k+1}\}_{k\in\N}$ are another classical family of snarks. The snark $G_{2k+1}$ is constructed in the following way. Let $P^-$ be the Petersen graph minus two vertices at distance $2$, take $2k+1$ copies $P_1^-,\dots,P_{2k+1}^-$ of $P^-$ and glue them together as shown in Figure~\ref{Fig:goldberg}. 

\begin{figure}[h]
		\centering
		\includegraphics[width=10cm]{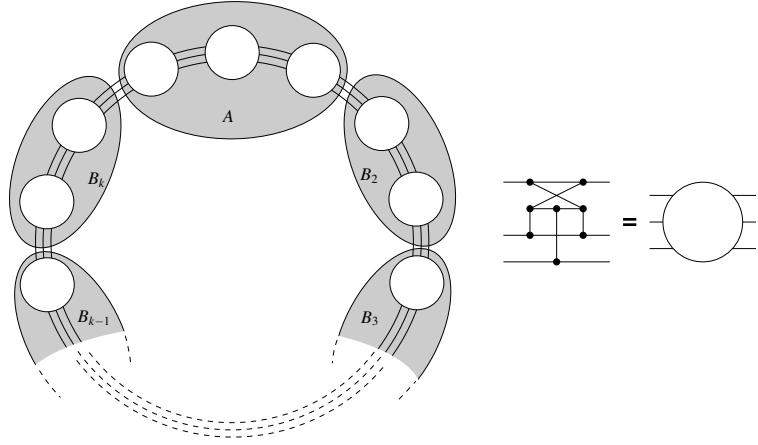}
	\caption{The Goldberg's snark $G_{2k+1}$ on $8(2k+1)$ vertices.}\label{Fig:goldberg}
\end{figure}

In~\cite{Luk} the circular flow number of the Goldberg snark $G_{2k+1}$ is shown to be inside the interval $[4+\frac{1}{2k+1}, 4+\frac{1}{k}]$. 
By using the algorithm described in Section~\ref{sec:algorithm}, we have shown (see Table~\ref{table:cfn_famous_snarks}) that $\Phi_c(G_{2k+1})=4+\frac{1}{k+1}$ for $k=1,2,3$. 
Here, we show that the value $4+\frac{1}{k+1}$ is an upper bound for the circular flow number of $G_{2k+1}$ for all possible $k$, thus improving the best known upper bound. 

\begin{proposition}\label{prop:goldberg}
Let $G_{2k+1}$ be the Goldberg snark of order $8(2k+1)$. Then, $$\Phi_c(G_{2k+1}) \leq 4+\frac{1}{k+1}.$$  
\end{proposition}
\begin{proof}
The Goldberg snark $G_{2k+1}$ consists of $2k+1$ copies of the Petersen graph minus two vertices at distance $2$ glued together as shown in Figure~\ref{Fig:goldberg}. Define the multipole $A$ to be three consecutive blocks of $G_{2k+1}$ and each multipole $B_t$ to be two consecutive blocks of $G_{2k+1}$, for $t=2,\dots, k.$ 
We define on these multipoles the nowhere-zero circular flow represented in Figures~\ref{fig:mult_A} and~\ref{fig:mult_B}. Note that the flow value on each edge of these multipoles is between $1$ and $3+\frac{1}{k+1}$. 

\begin{figure}[htb!]
			\centering
			\includegraphics[width=15cm]{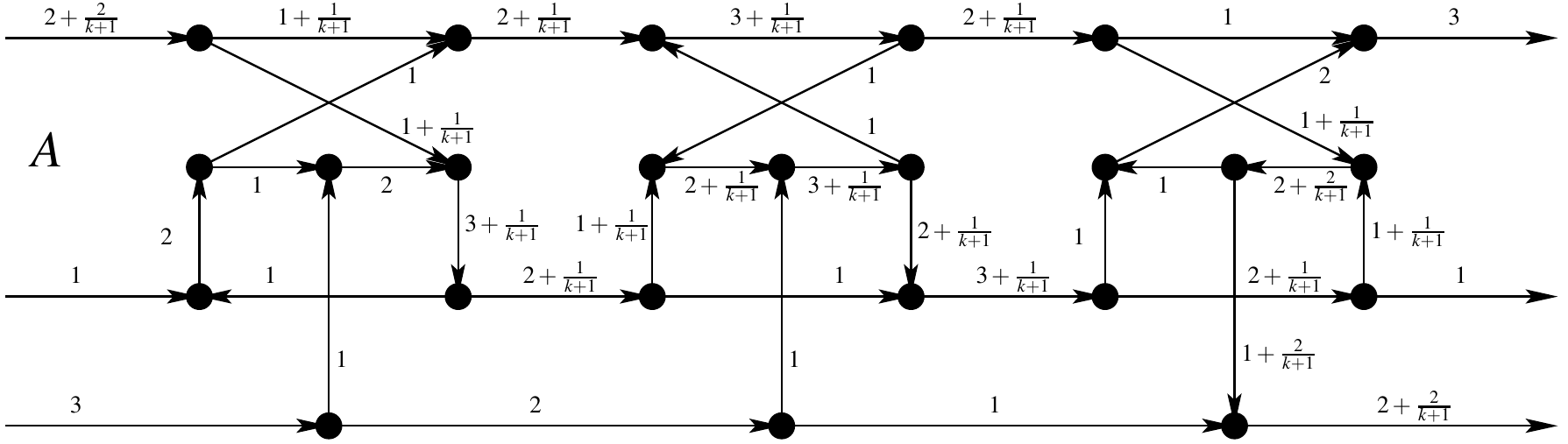}
			\caption{Flow in the multipole $A$.}\label{fig:mult_A}
\end{figure}

\begin{figure}[htb!]
	\centering
	\includegraphics[width=10cm]{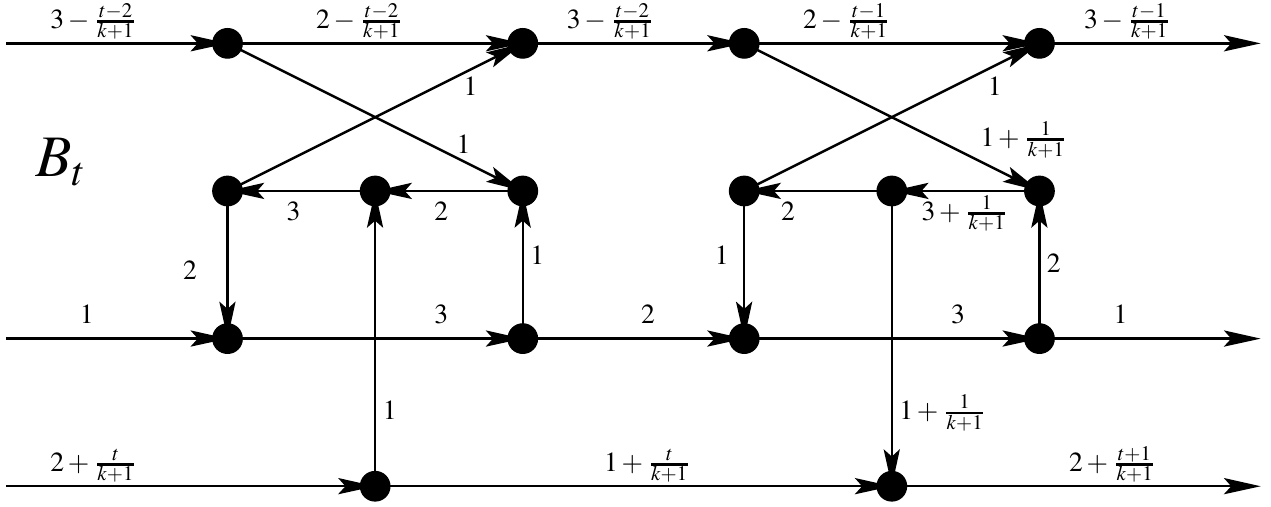}
	\caption{Flow in the multipole $B_{t}$.}\label{fig:mult_B}
\end{figure}

Moreover, we can glue them together as shown in Figure~\ref{Fig:goldberg}, in such a way that the Goldberg snark $G_{2k+1}$ is constructed. It follows that a nowhere-zero circular $(4+\frac{1}{k+1})$-flow is defined in $G_{2k+1}$, for every positive integer $k$.
\end{proof}

\section{A Certificate for non-Cubic Graphs with Circular Flow Number at Least 5} \label{Sec:certificatecircularflownumber5}

In general the problem of establishing the circular flow number of a bridgeless graph is hard to solve. Many flow problems indeed can be reduced to the class of cubic graphs where they can be attacked much more powerfully by making use of known structural properties of cubic graphs. For instance, it is a general result that every bridgeless graph has a nowhere-zero $k$-flow if and only if every bridgeless cubic graph has one. The still open $5$-flow Conjecture from Tutte claims that every bridgeless graph has a nowhere-zero $5$-flow. This puts interest in studying the structure of those graphs having circular flow number (at least) $5$, see for instance~\cite{EMT16} and~\cite{GMM}. In this section we show that the problem of deciding whether $\Phi_c(G)\ge 5 $ for a graph $G$ can be reduced to deciding whether $\Phi_c(H)\ge 5$ for every $H\in \ca H$, where $\ca H $ is a finite class of cubic graphs that can be constructed by applying suitable operations to $G$.

\begin{definition}	An {\bf expansion} of a vertex $v \in V(G)$ into a graph $K$ is obtained by the replacement of $v$ by $K$ and by adding as many edges  with one end in $V(G)-\{v\}$ and the other end in $V(K)$ as the degree of $v$ in $V(G)$.
\end{definition}

\begin{remark}\label{expansions}
It is easy to prove that each expansion of $G$ has circular flow number larger than or equal to the circular flow number of $G$.
\end{remark}

\begin{definition}
	Let $G$ be a graph and $v\in V(G)$ with $d_G(v)\ge4$. Define $\ca H^v(G)$ to be the class of all graphs that can be obtained from $G$ by expanding $v$ into a copy of $K_2$, in such a way that no vertex of degree $1$ or $2$ is created in the resulting graph.
\end{definition}

Note that, if we denote by $d$ the degree of $v\in G$, then \[ |\ca H^v(G)| \le 2^{d-1}-d-1. \] Indeed a graph in $\ca H^v(G)$ can be constructed by partitioning the set $\partial(v)$ of all edges incident to $v$ into two disjoint subsets $A_1$ and $A_2$ such that both of them consist of at least two elements and letting all edges of $A_1$ be adjacent to one vertex of $K_2$ and all edges of $A_2$ be adjacent to the other one. This way we can notice that the number of graphs in $\ca H^v(G)$ can be at most half of the number of subsets $A_1$ of edges incident to $v$, such that $|A_1|\in \{ 2,3,\dots,d-2 \}$. Thus \[ |\ca H^v(G)| \le \frac{2^d-2d-2}{2} = 2^{d-1}-d-1. \]

It may happen that $\ca H^v(G)$ contains some graphs with a bridge, even when $G$ is bridgeless. We recall that the circular flow number of a graph with a bridge is set to be~$\infty$.

Now suppose that $G$ has a vertex of degree $4$ and let $\ca C^v(G)$ be the class of all graphs that can be obtained by expanding $v\in G$ in one of the two ways shown in Figure~\ref{fig:expansions} (in such a way that no vertex of degree $2$ is created).

\begin{figure}[htb!]
			\centering
			\includegraphics[width=12cm]{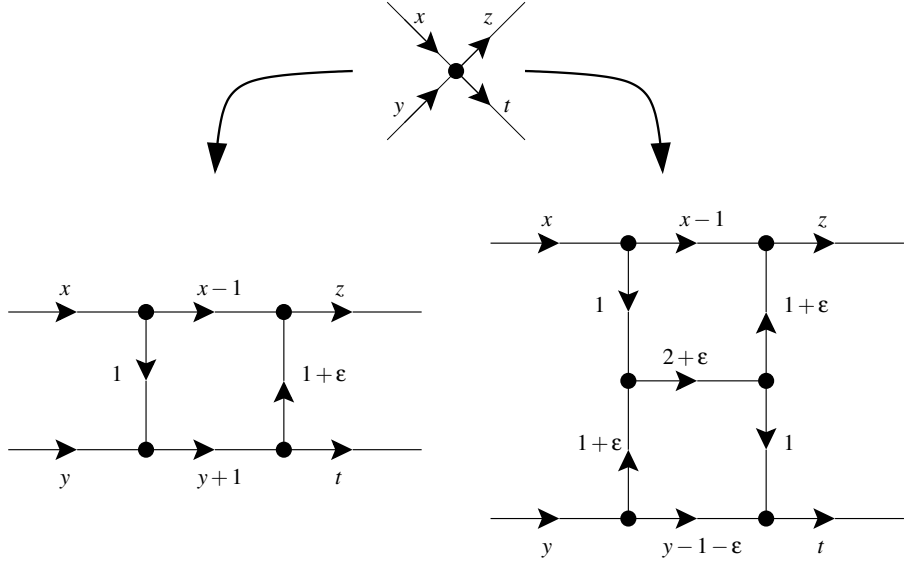}
	\caption{Expansions for a degree $4$ vertex.}\label{fig:expansions}
\end{figure}

\begin{definition}
	Let $G$ be a graph and $v\in V(G)$ with $d_G(v)=4$. Define \[ \ca G^v(G):= \ca C^v(G) \cup \ca H^v(G). \]
\end{definition}

The following two lemmas will be crucial for proving the main theorem of this section.

\begin{lemma}
	Let $G$ be a graph with a vertex $v$ of degree $4$. Then \[ \Phi_c(G)\ge 5 \iff \Phi_c(H)\ge5, \forall H \in \ca G^v(G). \]
\end{lemma}

\begin{proof}
	The implication from the left to the right follows by noticing that every graph in $\ca G^v(G)$ is an expansion of $G$ and by Remark~\ref{expansions}.
	
	For the other direction we prove that if $G$ has a circular $r$-flow $\phi$ with $r<5$ then $\phi$ can be extended to a circular $s$-flow with $s<5$ in at least one of the graphs of the family $\ca G^v(G)$.
	
	Take such an $r$-flow and orient $G$ in such a way that all flow values are positive. Call $x,y,z$ and $t$ the flow values of the four edges incident to $v$ and let $e_x,e_y,e_z$ and $e_t$ the corresponding edges.
	
	Suppose first that there are three incoming flow values at $v$, say $x,y$ and $z$. Then $x+y+z=t \in [1,4)$, meaning that $x+y\in [2,3)$. Then if we expand $v$ into a copy of $K_2$ with vertices $v_1$ and $v_2$, in such a way that $e_x,e_y$ are adjacent to $v_1$ and $e_z,e_t$ to $v_2$, then we can orient the new edge from $v_1$ to $v_2$ and assign it the flow value $x+y=t-z$. Therefore the flow has been extended properly.
	
	Suppose, on the other hand, that there are two incoming flow values $x,y$ and two outgoing ones $z,t$ (without loss of generality let $x\ge y$).
	If either $x+y \in [2,4)$ or $|x-z|\in [1,4)$ then we can repeat the argument used above, i.e.\ in the first case let $e_x,e_y$ be adjacent to $v_1$ and $e_z,e_t$ to $v_2$, whereas, in the second case, let $e_x,e_z$ be adjacent to $v_1$ and $e_y,e_t$ to $v_2$. The $r$-flow in $G$ naturally extends to an $s$-flow in this new graph, where $s<5$.
	
	Hence we can suppose that $x+y\ge 4$ and $z=x+\epsilon$ for a suitable $\epsilon\in [0,1)$ without loss of generality. Indeed, if $z\le x$ then we can just switch the orientation of all edges of $G$ and keep unchanged their flow value. Note that from $z=x+\epsilon$ we get that $t=y-\epsilon$. Furthermore, since $x+y\ge4 $ at least one between $x$ and $y$ must be at least $2$. By the assumption $x\ge y$ it must be $x\ge2$. Therefore, if $y<3$ we extend the flow as shown in the left part of Figure~\ref{expansions}, while, if $y\ge3$, we extend the flow as shown in the right part of Figure~\ref{expansions}.
\end{proof}

\begin{lemma}
	Let $G$ be a graph with a vertex $v$ of degree at least $5$. Then \[ \Phi_c(G)\ge 5 \iff \Phi_c(H)\ge5, \forall H \in \ca H^v(G). \]
\end{lemma}

\begin{proof}
	The implication from the left to the right follows by noticing that every graph in $\ca H^v(G)$ is an expansion of $G$ and by Remark~\ref{expansions}.
	
	For the other direction we prove that if $\Phi_c(G) < 5$ then there is at least one graph in $\ca H^v(G)$ that has circular flow number less than $5$.
	
	So take a circular nowhere-zero $\Phi_c(G)$-flow $\phi$ in $G$ with an orientation such that all flow values are positive. Let $n_1:=|E^-(v)|>0$ the number of incoming edges at $v$ and let $x_1,\dots,x_{n_1} \in [1,4)$ be all incoming flow values at $v$. Respectively, let $n_2:=|E^+(v)|>0$ and $y_1,\dots,y_{n_2} \in [1,4)$ be all outgoing flow values at $v$. Without loss of generality we can suppose $n_1\ge n_2$, because the entire orientation can be reversed. Moreover let us denote by $e_{x_i}$ the incoming arc at $v$ whose flow value is $x_i$, and equivalently $e_{y_j}$ the outgoing arc at $v$ whose flow value is $y_j$.
	
	Let us first prove the statement in the case of degree exactly $5$.
	
	Notice that there cannot be $4$ incoming edges at $v$ since their sum should be at the same time greater then or equal to $4$ and equal to the outgoing flow value, which is less than $4$.
	
	Hence there are three incoming flow values, namely $x_1,x_2 $ and $x_3$, and two outgoing flow values, namely $y_1$ and $y_2$. Let \[ r:= x_1+x_2+x_3=y_1+y_2. \]
	
	Notice that, in order to complete the proof it is enough to exhibit a partition of $E(v) = E_1\cup E_2$ into two disjoint subsets $E_1$ and $E_2$ of edges such that both $|E_i|>1$ (condition required in order to guarantee that no vertices of degree $1$ or $2$ appear) and
	\begin{equation}
		|\sum_{e\in E_1\cap E^-(v)} \phi(e) - \sum_{e\in E_1\cap E^+(v)} \phi(e)| = |\sum_{e\in E_2\cap E^+(v)} \phi(e) - \sum_{e\in E_2\cap E^-(v)} \phi(e)| \in [1,4).
	\end{equation}
	Indeed, if $H\in\ca H^v(G) $ is obtained by expanding $v$ into the new edge $v_1v_2$ and letting all edges of $E_i$ be adjacent only to $v_i$, we can extend the flow $\phi$ in $G$ to a suitable $s$-flow in $H$ ($s<5$) by letting \[\phi(v_1v_2) := \sum_{e\in E_1\cap E^-(v)} \phi(e) - \sum_{e\in E_1\cap E^+(v)} \phi(e),\]
	where $v_1v_2$ is oriented from $v_1$ to $v_2$.
	
	If there are $x_i$ and $x_j$, $i\ne j$, such that $x_i+x_j< 4$ then we are done, by taking $E_1 = \{ e_{x_i},e_{x_j} \}$. Hence, we can suppose that \[ \begin{cases}
	x_1+x_2  \ge4 \\
	x_2+x_3 \ge4 \\
	x_1 +x_3 \ge 4
	\end{cases} \]
	and by summing all those three equations together we get \[ r\ge 6. \]
	
	There exists an incoming flow value, say $\hat{x}$, which is at most $\frac{r}{3}$, and there exists an outgoing flow value, say $\hat{y}$, which is at least $1+\frac{r}{3}$. Otherwise, if both $y_i$s are less than $ 1+\frac{r}{3}$, we get the following contradiction: \[ r = y_1+y_2 < 2\bigl(1 + \frac{r}{3} \bigr ) = 2 + \frac{2}{3}r \le \frac{1}{3}r + \frac{2}{3}r = r. \]
	Therefore $\hat{y}-\hat{x} \in [1,4)$ and if we take $E_1=\{ \hat{x},\hat{y} \}$ we obtain a graph in $\ca H^v(G)$ with circular flow number less than $5$.
	
	Let us now argue for the case $d_G(v)\ge 6$. If \[ \sum_{i\in I} x_i - \sum_{j\in J} y_j \in [1,4), \] for two suitable subsets $I\subseteq \{1,\dots, n_1\}, J \subseteq \{ 1,\dots,n_2 \}$, such that $1<|I|+|J|<n_1+n_2-2$, then the thesis follows by taking $E_1=\{ e_{x_i}\colon i\in I\} \cup \{e_{y_{j}}\colon  j\in J \}$. 
	
	Now let \[r:=\sum_{i=1}^{n_1} x_i = \sum_{j=1}^{n_2} y_j,\] be the sum of all incoming (or all outgoing) flow values and define \[
	\begin{cases}
	r_1 := \frac{r}{n_1} \quad \quad \text{ mean incoming flow value,} \\
	r_2 := \frac{r}{n_2}  \quad \quad \text{ mean outgoing flow value.}
	\end{cases}\]
	
	Suppose that $r_2-r_1 \ge 1$. Then there exists $x_{\alpha}$ and $y_{\beta}$ such that \[ x_{\alpha} +1 \le r_1 +1 \le r_2 \le y_{\beta}, \]
	and we are done again by setting $E_1 =\{e_{x_{\alpha}},e_{y_{\beta}}\}$.
	
	Suppose that $r_2-r_1<1$ and \[|\sum_{i\in I} x_i - \sum_{j\in J} y_j | \notin [1,4), \] for every $I\subseteq \{1,\dots, n_1\}, J \subseteq \{ 1,\dots,n_2 \}$, with $1<|I|+|J|<n_1+n_2-2$. We show that this set of hypothesis leads to a contradiction.
	
	There are $x_{a}$ and $x_{b}$ such that $x_a+x_b \le 2r_1 \le 2r_2$. Hence, if $x_a+x_b \in [5,2r_2]$ then take $\tilde{y}\in [r_2,4)$ and notice that \[x_a+x_b-\tilde{y}\in (1,r_2],\] which is not possible.
	
	Therefore, for every couple of different incoming flow values $x_a$ and $x_b$ such that $x_a+x_b \le 2r_1$, the sum $x_a+x_b<5$, meaning that $x_a+x_b \in [4,5)$. Then at least one of them, let us say $x_a$, is such that $x_a\in [1,2.5)$. As a consequence, since $|x_a-y_j|<1$, for every $j$, we get that \[ y_j\in [1,3.5),\forall j\in \{1,\dots, n_2\}. \]
	\begin{itemize}
		\item If there is a $\hat{y}\in [1,3]$ then $x_a+x_b-\hat{y}\in [1,4)$, a contradiction.
		\item If, on the other hand, every $y_j\in (3,3.5)$ then $y_1+y_2\in (6,7)$ and so $y_1+y_2-(x_a+x_b)\in (1,3)$, a contradiction again. \qedhere
	\end{itemize}
\end{proof}

The previous lemmas lead to the following theorem, which is the main result of this section.

\begin{theorem}\label{circflow5}
	Let $G$ be a bridgeless graph. There exists a finite family of bridgeless cubic graphs $\ca H$ such that \[ \Phi_c(G)\ge 5 \iff \Phi_c(H)\ge5 \quad \forall H\in \ca H. \]
\end{theorem}

\section{Open Problems and New Conjectures}
\label{sec:open_problems}

We conclude our paper with some open problems and conjectures.

\begin{itemize}

\item We verified by computer that none of the cyclically $4$-edge-connected cubic graphs $G$ without a $3$-edge-coloring, having girth at least $4$ and on at most $32$ vertices such that $|V(G)|\equiv 0$ or $6 \mod 8$ has a circular flow number that attains the lower bound of Lukot'ka and \v Skoviera~\cite{Luk_cycle} from Theorem~\ref{thm:luksko} (cf.\ Table~\ref{table:cfn_values}). This fact implies that none of the non-$3$-edge-colorable cubic graphs of order at most $32$ such that $|V(G)|\equiv 0$ or $6 \mod 8$ has a circular flow number attaining the bound from Theorem~\ref{thm:luksko}. Indeed, assume that there exists such a graph $G$ having a $3$-edge-cut (and the property that $\Phi_c(G)$ attains the bound by Theorem \ref{thm:luksko}). Then, a non-$3$-edge-colorable smaller graph can be constructed by contracting one of the two sides of the $3$-edge-cut, say $H$, and $\Phi_c(H)\le \Phi_c(G)$ holds, because $G$ could be seen as an expansion of $H$. Hence, either we get a contradiction with Theorem~\ref{thm:luksko} if $H$ has much fewer vertices than $G$, or, iterating this process, we get a cyclically $4$-edge-connected cubic graph with no $3$-edge-coloring having circular flow number that attains the bound from Theorem~\ref{thm:luksko}, in contradiction with our computational results. Note that a similar argument applies to $2$-edge-cuts.

Computational evidence suggests the following strengthened version of Theorem~\ref{thm:luksko}:

\begin{conjecture}
	Let $G$ be a connected bridgeless cubic graph of order at most $8k+8$ that does not admit any $3$-edge-coloring. Then \[ \Phi_c(G) \ge 4+\frac{1}{k}. \]
\end{conjecture}

\item By using the algorithm presented in this paper, we verified that the circular flow number of the Goldberg snarks $G_3$, $G_5$ and $G_7$ meet the upper bound given by Proposition~\ref{prop:goldberg}. This seems to suggest that the following conjecture could be true:

\begin{conjecture}
	Let $G_{2k+1}$ be the Goldberg's snark on $8(2k+1)$ vertices. Then \[\Phi_c(G_{2k+1})=4+\frac{1}{k+1}\] for every positive integer $k$.
\end{conjecture}

\item In Section~\ref{Sec:certificatecircularflownumber5} we show that the problem of deciding if a graph has circular flow number $5$ or not can be reduced to cubic graphs. More in general, we showed that it can be reduced to checking a class of graphs with vertices of smaller degree. In doing this we tried to focus on the smallest possible expansions, but we do not pay special attention to the number of constructed graphs. Proving a version of Theorem~\ref{circflow5} with some additional restrictions on the cardinality of the family $\cal{H}$ could be an interesting problem.

\end{itemize}

\section*{Acknowledgements}
Most computations for this work were carried out using the Stevin Supercomputer Infrastructure at Ghent University.

\end{document}